\documentclass[12pt]{amsart}
\usepackage[all]{xy}

\usepackage{graphicx}  

\usepackage{amsmath}
\usepackage{amssymb}
\usepackage{amsfonts}
\usepackage{mathrsfs}
\usepackage{mathtools}

\newcommand{\Cdb}{\ensuremath{\mathbb{C}}}
\newcommand{\Ddb}{\ensuremath{\mathbb{D}}}

\newcommand{\Pdb}{\ensuremath{\mathbb{P}}}

\newcommand{\Rdb}{\ensuremath{\mathbb{R}}}

\newcommand{\M}{\mbox{${\mathcal M}$}}

\newcommand{\norm}[1]{\Vert#1\Vert}
\newcommand{\bignorm}[1]{\bigl\Vert#1\bigr\Vert}
\newcommand{\Bignorm}[1]{\Bigl\Vert#1\Bigr\Vert}

\newtheorem{theorem}{Theorem}[section]
\newtheorem{lemma}[theorem]{Lemma}
\newtheorem{corollary}[theorem]{Corollary}
\newtheorem{proposition}[theorem]{Proposition}

\theoremstyle{remark}

\theoremstyle{definition}

\numberwithin{equation}{section}

\begin{document}

\title[]{$\alpha$-admissibility for Ritt operators}

\author{Christian Le Merdy}
\address{Laboratoire de Math\'ematiques\\ Universit\'e de  Franche-Comt\'e
\\ 25030 Besan\c con Cedex\\ France}
\email{clemerdy@univ-fcomte.fr}

\date{\today}

\thanks{The author is supported by the research program ANR 2011 BS01 008 01}

\begin{abstract} Let $T\colon X\to X$ be a power bounded operator on Banach space. An 
operator $C\colon X\to Y$ is called admissible for $T$ if it satisfies
an estimate $\sum_k\norm{CT^k(x)}^2\,\leq M^2\norm{x}^2$. Following Harper and Wynn,
we study the validity of a certain Weiss conjecture in this discrete setting.
We show that when $X$ is reflexive and $T$ is a Ritt operator satisfying a appropriate
square function estimate, $C$ is admissible for $T$ if and only if it satisfies a uniform
estimate $(1-\vert \omega\vert^2)^{\frac{1}{2}}\norm{C(I-\omega T)^{-1}}\,\leq K\,$ for
$\omega\in \Cdb$, $\vert\omega\vert<1$. We extend this result to the more general 
setting of $\alpha$-admissibility. Then we investigate a natural variant 
of admissibility involving $R$-boundedness and provide examples to which our general results apply.
\end{abstract}

\maketitle


\bigskip

\section{Introduction}
Admissibility of observation operators for semigroups has attracted 
a lot of attention in the last fifteen years. In general, one starts
with a bounded $c_0$-semigroup $(T_t)_{t\geq 0}$ on some Banach space
$X$, with generator $-B$, one considers an operator $C$ defined and
continuous on the domain of $B$, taking values in another Banach space $Y$
($C$ is the so-called observation operator) 
and one wonders whether there exists an estimate
$$
\int_{0}^{\infty}\bignorm{CT_t(x)}^2\, dt\ \leq\, M^2\norm{x}^2
$$
valid for any $x$ belonging to the domain of $B$. In this case 
$C$ is called admissible for $(T_t)_{t\geq 0}$, and this property 
is important for the study of certain linear control systems.
A general question then is to find criteria ensuring admissibility.
We refer the reader to \cite{JP} for background, various results around this
question, 
and applications.

In the paper \cite{LM2}, we exhibited such a criterion by showing that
if $(T_t)_{t\geq 0}$ is a bounded analytic semigroup satisfying a 
square function estimate 
\begin{equation}\label{1London}
\int_{0}^{\infty}\bignorm{B^{\frac{1}{2}}
T_t(x)}^2\, dt\ \leq\, \kappa^2\norm{x}^2,
\end{equation}
then an observation
operator $C$ is admissible for $(T_t)_{t\geq 0}$ if and only if
there exists a constant $K\geq 0$ such that $t^{\frac{1}{2}}\norm{C(t+B)^{-1}}\leq K\,$
for any $t>0$. The latter is a variant of the so-called Weiss condition.
The above result is especially interesting when $X$ is a Hilbert space. Indeed
in this case, the square function estimate (\ref{1London}) holds true if $\norm{T_t}\leq 1$
for any $t\geq 0$.

Recently Harper \cite{H} investigated a notion of discrete admissibility, in connection
with discrete time control systems. One starts with a power bounded operator 
$T$ on $X$, a bounded operator $C$ from $X$ into $Y$ and one says that 
$C$ is admissible 
for $T$ if there is a constant $M\geq 0$ such that
$$
\sum_{k=0}^{\infty}\norm{CT^k(x)}^2\,\leq M^2\norm{x}^2,\qquad x\in X.
$$
As in the continuous case, the main issue is to find spectral conditions 
or resolvent estimates ensuring that property.
This topic was developed by Wynn in a series of recent papers 
\cite{W1, W2, W3}. Our main purpose is to establish a new criterion for that discrete
admissibility, in the spirit of the above cited result from \cite{LM2}. 

Namely, 
let $\Ddb=\{\omega\in\Cdb\, :\, \vert\omega\vert<1\}$ be the open unit disc 
of the complex plane. We will show that if $T$ is a Ritt operator 
and there exists a constant $\kappa\geq 0$
such that 
\begin{equation}\label{1SFE}
\sum_{k=1}^{\infty} k\bignorm{T^k(x)-T^{k-1}(x)}^2\,\leq \kappa^2\norm{x}^2,
\qquad x\in X,
\end{equation}
then $C$ is admissible for $T$ if and only if there exists a constant 
$K\geq 0$ such that
\begin{equation}\label{1Weiss}
\bigl(1-\vert \omega\vert^2\bigr)^{\frac{1}{2}}\bignorm{C(I-\omega T)^{-1}}\,\leq K,
\qquad \omega\in\Ddb.
\end{equation}
Again this is particularly interesting when $X$ is a Hilbert space, since 
(\ref{1SFE}) holds true whenever $T$ is a Ritt operator and $\norm{T}\leq 1$
(see \cite[Thm. 8.1]{LM3}). Thus in the terminology of 
\cite{H, W1, W2, W3}, Ritt contractions on Hilbert space satisfy a strong form
of the discrete
Weiss conjecture.

The above results will be established in Section 3, however we place
them in a broader context. We consider 
a weighted form of admissibility, called $\alpha$-admissibility. That concept was
introduced in \cite{HL} in the classical semigroup setting and then in \cite{W1}
in the discrete setting. See the above two papers for motivation. All the
necessary background is provided in Section 2. In this more general framework, we 
show that under the assumption (\ref{1SFE}), $\alpha$-admissibility is 
equivalent to a certain resolvent estimate.

In the last Section 4, we consider a variant of the above discussed notion, called $R$-admissibility.
This name refers to $R$-boundedness. The study of $R$-admissibility in the semigroup setting 
goes back to \cite{H, HK1, HK2, LM2}, and turns out to have more applications than
classical admissibility when one deals with non Hilbertian Banach spaces.
We  introduce a relevant definition of discrete $R$-admissibility and again, we give 
sufficient conditions under which $R$-admissibilty can be characterized by
a resolvent estimate.

\medskip
\section{$\alpha$-admissibility, resolvent estimates and square functions}
We start with a few definitions and preliminary results which either
come from Wynn's paper \cite{W1} or generalize it in a simple way.

Let $X$ be a Banach space and let $T\colon X\to X$ be a power bounded operator, that is,
there exists a constant $c_0\geq 0$ such that 
\begin{equation}\label{2Pb}
\norm{T^k}\leq c_0,\qquad k\geq 0.
\end{equation}
Let $Y$ be another Banach space, let $C\colon X\to Y$ be a bounded operator and let
$\alpha>-1$ be a real number. We say that $C$ 
is $\alpha$-admissible for $T$ if there is a constant $M\geq 0$ such that
\begin{equation}\label{2Admiss}
\sum_{k=0}^{\infty} (k+1)^{\alpha} \norm{CT^k(x)}^2\,\leq M^2\norm{x}^2,\qquad x\in X.
\end{equation}
Admissibility (as discussed in the introduction) simply coincides with $0$-admissibility.

The following elementary result connects $\alpha$-admissibility to resolvent
estimates. Note that the power boundedness assumption ensures that 
the spectrum of $T$ is included in $\overline{\Ddb}$.

\begin{proposition}\label{2Easy} Let $\alpha,\beta>-1$ be real numbers 
such that $m=\frac{\alpha+\beta}{2}$ is a nonnegative integer.
If $C$ is $\alpha$-admissible for $T$, then there exists a constant $K\geq 0$ such that
\begin{equation}\label{2Weiss}
\bigl(1-\vert \omega\vert^2\bigr)^{\frac{1+\beta}{2}}\bignorm{C(I-\omega T)^{-(m+1)}}\,\leq K,
\qquad \omega\in\Ddb.
\end{equation}
\end{proposition}

\begin{proof}
We assume that $T$ satisfies (\ref{2Admiss}). 
We will use the existence of a constant $c>0$ such that 
\begin{equation}\label{2c}
\sum_{k=1}^{\infty} k^{\beta} s^{k-1}\,\leq \,\frac{c}{(1-s)^{\beta+1}},\qquad s\in (0,1).
\end{equation}
See e.g. \cite[Lemma 2.2]{W1} for a proof of this uniform estimate.

For any integer $k\geq 1$, we set
\begin{equation}\label{2ck}
c_k=\frac{k(k+1)\cdots(k+m-1)}{k^{\frac{\alpha}{2}}},
\end{equation}
with the convention that the above numerator equals 1 when $m=0$. 
For any $z\in\Ddb$, we have
$$
\frac{m!}{(1-z)^{m+1}}\,=\,\sum_{k=1}^{\infty} k(k+1)\cdots(k+m-1)z^{k-1},
$$
hence for any $\omega\in \Ddb$ we have
\begin{equation}\label{2Exp}
m!(I-\omega T)^{-(m+1)}\,=\,\sum_{k=1}^{\infty} c_k \omega^{k-1}  k^{\frac{\alpha}{2}}T^{k-1}.
\end{equation}
Composing with $C$ and using the Cauchy-Schwarz inequality, we deduce that for any $x\in X$,
\begin{align*}
\bignorm{C(I-\omega T)^{-(m+1)}x}\, & \leq\,
\frac{1}{m!}\,\Bigl(\sum_{k=1}^{\infty} k^{\alpha}
\norm{CT^{k-1}(x)}^2\Bigr)^{\frac{1}{2}}\,
\Bigl(\sum_{k=1}^{\infty} c_k^2\vert\omega\vert^{2(k-1)}\Bigr)^{\frac{1}{2}}\\
&\leq \,\frac{M}{m!}\, \norm{x}
\Bigl(\sum_{k=1}^{\infty} c_k^2\vert
\omega\vert^{2(k-1)}\Bigr)^{\frac{1}{2}}.
\end{align*}
Now observe that $c_k\sim_{\infty} k^{m-\frac{\alpha}{2}}$, so that $c_k^2\sim_{\infty} k^{\beta}$.
Hence for a suitable constant $M'\geq 0$, we actually have an estimate
$$
\bignorm{C(I-\omega T)^{-(m+1)}x}\,\leq M' \norm{x}
\Bigl(\sum_{k=1}^{\infty} k^{\beta}\vert
\omega\vert^{2(k-1)}\Bigr)^{\frac{1}{2}}.
$$
Applying (\ref{2c}) with $s=\vert\omega\vert^2$, we deduce (\ref{2Weiss}) with $K=c^{\frac{1}{2}}M'$.
\end{proof}

We will focus on the specific class of Ritt operators.
We shall use results and ideas from \cite{ALM, LM3} to which
we refer for background and relevant information. We recall
that a power bounded operator $T\colon X\to X$ is called a Ritt operator if
there exists a constant $c_1\geq 0$ such that 
\begin{equation}\label{2Ritt}
k \bignorm{T^k - T^{k-1}}\leq c_1,\qquad k\geq 1.
\end{equation}
In this case, the operator $I-T$ is sectorial and we 
may therefore define its fractional powers $(I-T)^a$ for any $a>0$
(see e.g. \cite[Sect. 2]{LM3} and \cite[Chap. 3]{Ha} for details).
Then we define a `square function' $\norm{\ \cdotp}_{T,a}$ on $X$ by letting
\begin{equation}\label{2SF}
\norm{x}_{T,a} =\,\Bigl(\sum_{k=1}^{\infty} k^{2a-1}\bignorm{T^{k-1}(I-T)^{a}x}^2\Bigr)^{\frac{1}{2}},
\qquad x\in X.
\end{equation}
%Note that $\norm{x}_{T,1}$ is equal to the square root 
%of the left handside of (\ref{2SFE}).
Note that $\norm{x}_{T,a}$ may be infinite.
For the sake of clarity we indicate that when $X$
is a Hilbert space, the above square functions coincide with the ones defined 
in \cite{ALM, LM3}. However this is no longer the case when $X$ is not Hilbertian.
In that situation, the square functions from \cite{ALM, LM3} coincide with the ones which 
will be considered later on in Section 4.

When $X$ is reflexive, the Mean Ergodic Theorem ensures that we have a direct sum decomposition
\begin{equation}\label{2MET}
X={\rm Ker}(I-T)\oplus \overline{{\rm Ran}(I-T)}.
\end{equation}
Then the argument in the proof of \cite[Thm. 3.3]{ALM} shows that the
above square functions are pairwise equivalent. We state that result for further 
use.

\begin{lemma}\label{2Indep}
Assume that $X$ is reflexive and
let $T\colon X\to X$ be a Ritt operator. For any  positive real numbers $a,a'>0$,
there exists a constant $c>0$ such that 
$$
c^{-1}\norm{x}_{T, a'}\leq \norm{x}_{T,a} \leq c\norm{x}_{T,a'}
$$
for any $x\in X$.
\end{lemma}

In the sequel we say that $T$ satisfies a square function estimate if it satisfies 
an inequality $\norm{x}_{T,a}\leq \kappa\norm{x}$ for one (equivalently for all)
$a>0$. To be more specific ($a=1$), $T$ satisfies a square function estimate 
when there exists a constant $\kappa\geq 0$
such that 
\begin{equation}\label{2SFE}
\sum_{k=1}^{\infty} k\bignorm{T^k(x)-T^{k-1}(x)}^2\,\leq \kappa^2\norm{x}^2,
\qquad x\in X.
\end{equation}

\medskip
\section{Ritt operators satisfying the discrete Weiss conjecture}
The general question considered in this main section 
is whether the estimate (\ref{2Weiss}) implies that $C$
is $\alpha$-admissible for $T$. That question actually has several
variants. In \cite{W1,W2,W3}, the case when
$\alpha\in(-1,1)$, $\beta=-\alpha$ and $X,Y$ are Hilbert spaces is considered. In this
situation, $T\colon X\to X$ is said to satisfy the discrete Weiss conjecture 
if any $C\colon X\to Y$ satisfying (\ref{2Weiss}) is $\alpha$-admissible for $T$.
This is shown to be the case when $\alpha\in (0,1)$, $Y=\Cdb$ and $T$ is normal.
Moreover it is shown in \cite{H} that when $\alpha=0$ and $Y=\Cdb$, any contraction
$T$ on Hilbert space also satisfies the discrete Weiss conjecture. As far as
we know, all other published results are counterexamples. 
In the case $Y=\Cdb$, it is shown in \cite{W2}
that for any $\alpha\in(-1,0)$, there exist normal operators not satisfying 
the discrete Weiss conjecture, whereas it is shown in \cite{W3}
that for any $\alpha\in(0,1)$, there exist contractions not satisfying 
the discrete Weiss conjecture.

The main interest of the next result (and of its extensions in Section 4) is to provide
positive results, and hence large classes of operators satisfying a discrete Weiss conjecture.
In particular Corollary \ref{3Cont} should be compared to the above mentioned results.

\begin{theorem}\label{3Main}
Let $X,Y$ be Banach spaces, assume that $X$ is reflexive and let $T\colon X\to X$ be a Ritt operator.
Assume that $T$ satisfies the square function estimate (\ref{2SFE}). Let $\alpha>-1$ and
$\beta\in (-1,3)$ be real numbers such that $m=\frac{\alpha+\beta}{2}$
is a nonnegative integer. Then a bounded operator $C\colon X\to Y$ is
$\alpha$-admissible for $T$ if (and only if) there exists a constant $K\geq 0$ such that
\begin{equation}\label{3Weiss}
\bigl(1-\vert \omega\vert^2\bigr)^{\frac{1+\beta}{2}}\bignorm{C(I-\omega T)^{-(m+1)}}\,\leq K,
\qquad \omega\in\Ddb.
\end{equation}
\end{theorem}

Before writing the proof of this theorem, we need some background on sectorial
operators. For any $\nu\in(0,\pi)$, we let
$$
\Sigma_\nu=\bigl\{\xi\in\Cdb^*\, :\, \vert{\rm Arg}(\xi)\vert<\nu\bigr\}
$$
be the sector of angle $2\nu$ around the positive real axis.
By definition a closed operator $B\colon D(B)\to X$ with dense domain $D(B)\subset X$
is called sectorial if there exists an angle 
$\nu\in(0,\pi)$ such that $\xi -B$ is invertible for any $\xi\in\Cdb\setminus \overline{\Sigma_\nu}$
and there exists a constant $K_\nu\geq 0$ such that
$$
\vert \xi\vert \norm{(\xi -B)^{-1}}\leq K_\nu,\qquad
\xi\notin \overline{\Sigma_\nu}.
$$
We say that $B$ is of type $\sigma\in (0,\pi)$ if this holds true for any $\nu\in (\sigma,\pi)$.

In the sequel, we let
$$
\varphi_{\theta}(z) = \frac{z^{\theta}}{1+z}, \qquad z\in\Cdb\setminus\Rdb_-,
$$
for any $\theta\in(0,1)$. The holomorphic functional calculus and the definition
of fractional powers of sectorial operators lead to
$$
\varphi_\theta(B) =  B^{\theta}(I+B)^{-1},
$$
see e.g. \cite{Ha} for details. 

Any sectorial operator with a dense range is 1-1 and in this case, the resulting operator
$B^{-1}$ is a well-defined closed operator whose domain is equal to the range of $B$
(see \cite[Thm. 3.8]{CDMY}).

Another point is that whenever $T\colon X\to X$ is a Ritt operator, then
the spectrum of $T$ is included in $\Ddb\cup\{1\}$ and there exists a constant $c_2\geq 0$
such that 
\begin{equation}\label{3Ritt}
\norm{(\lambda-T)^{-1}}\leq\,\frac{c_2}{\vert \lambda -1\vert}\,,\qquad 
\vert\lambda\vert>1,
\end{equation}
see e.g. \cite{Bl1}. This implies that 
$B=I-T$ is a (bounded) sectorial operator of type $<\frac{\pi}{2}$. We refer the reader  \cite{ Bl1, Bl2,LM3} for more on the
relationships between Ritt and sectorial operators.

Theorem \ref{3Main} is the discrete analog of \cite[Thm. 4.2]{HL}. In the course of the
proof of the latter `sectorial' result, we established the following property that we state 
for later use.

\begin{lemma}\label{3HLM}
Let $\alpha,\beta,m$ as in Theorem \ref{3Main} above, and
let $B\colon D(B)\to X$ be a sectorial operator. Assume that $B$ has a dense range.
Let $\Delta\colon D(B^{m+1})\to Y$ be a continuous operator such that the set
\begin{equation}\label{3HLM1}
\bigl\{t^{\frac{1+\beta}{2}} \Delta(t+B)^{-(m+1)}\, :\, t>0\bigr\}
\end{equation}
is bounded. Let $\theta\in (0,1)$ such that $\theta<\frac{1+\beta}{2}<1+\theta$.
Then the set
$$
\bigl\{t^{\frac{1+\alpha}{2} -m} \Delta B^{-m}\varphi_\theta(tB)\, :\, t>0\bigr\}
$$
is bounded as well.
\end{lemma}

\begin{proof}[Proof of Theorem \ref{3Main}] By Proposition \ref{2Easy} we only need to 
prove the `if' assertion.
Assume (\ref{3Weiss}) and  let $B=I-T$. We
fix some $\theta\in(0,1)$ such that $\theta<\frac{1+\beta}{2}<1+\theta$.
We aim at proving that the set (\ref{3HLM1}) for the operator $\Delta=C$
is bounded. Consider 
an arbitrary $t>0$ and set $\omega=\frac{1}{1+t}$. On the one hand, we have
$$
t+B = t+1 -T = (t+1)(I-\omega T),
$$
hence
$$
(t+B)^{-(m+1)} \,=\, \frac{1}{(t+1)^{m+1}}\,(I-\omega T)^{-(m+1)}.
$$
On the other hand, 
$$
\bigl(1-\vert \omega\vert^2\bigr)^{\frac{1+\beta}{2}}
= \Bigl(1 -\,\frac{1}{(1+t)^2}\Bigr)^{\frac{1+\beta}{2}} = \frac{(t^2+2t)^{\frac{1+\beta}{2}}}{(1+t)^{1+\beta}}
= t^{\frac{1+\beta}{2}}\,\frac{(2+t)^{\frac{1+\beta}{2}}}{(1+t)^{1+\beta}}.
$$
Consequently,
$$
t^{\frac{1+\beta}{2}} C(t+B)^{-(m+1)} = \bigl(1-\vert \omega\vert^2\bigr)^{\frac{1+\beta}{2}}
\,C(I-\omega T)^{-(m+1)}\,\rho(t),
$$
with 
$$
\rho(t) =(1+t)^{\beta-m}(2+t)^{-\frac{1+\beta}{2}}\,. 
$$
Since $\alpha+\beta=2m$, the order of
$\rho(t)$ at $\infty$ is $t^{-\frac{1+\alpha}{2}}$. Since $\alpha>-1$, this implies that
$\rho$ is bounded on $(0,\infty)$. This implies the boundedness of (\ref{3HLM1}).

For any $x\in {\rm Ker}(B)$, 
$$
t^{\frac{1+\beta}{2}}C(t+B)^{-(m+1)}(x)= t^{\frac{1+\beta}{2} -(m+1)}C(x) = t^{-\frac{1+\alpha}{2}}C(x).
$$
Since $\alpha>-1$, the boundeness of (\ref{3HLM1}) implies that $C(x)=0$ in this case. 
This shows that (\ref{2Admiss}) holds true on
${\rm Ker}(B)$. By the reflexivity assumption and (\ref{2MET}), it therefore suffices 
to show (\ref{2Admiss}) on $\overline{{\rm Ran}(B)}$. 

Then we may (and do) assume that 
$B=I-T$ has a dense range. Applying Lemma \ref{3HLM},
we deduce the existence of a constant $K'\geq 0$ such that
\begin{equation}\label{3K'}
t^{\frac{1+\alpha}{2} -m} \bignorm{C B^{-m}\varphi_\theta(tB)}\leq K',\qquad t>0.
\end{equation}

Assume that $m\geq 1$ (the case $m=0$ is treated later) . For any integer 
$k\geq 1$, we define an operator $U_k\colon X\to X$ by setting
$$
U_k = k^{m-\frac{1}{2} -\theta}(1+kB) B^{m-\theta} T^{k-1}.
$$
Then
\begin{align*}
k^{\frac{\alpha}{2}} C T^{k-1}
& = k^{\frac{\alpha+1}{2}-m} CB^{-m} k^\theta B^\theta (1+kB)^{-1} U_k\\
& = k^{\frac{\alpha+1}{2}-m} CB^{-m} \varphi_\theta(kB) U_k.
\end{align*}
Applying (\ref{3K'}), we deduce that for any $x\in X$,
\begin{equation}\label{3Sum}
\sum_{k=1}^{\infty} k^\alpha \norm{C T^{k-1}(x)}^2\,
\leq K'^2 \sum_{k=1}^{\infty} \norm{U_k(x)}^2.
\end{equation}
Now observe that 
\begin{equation}\label{3Uk}
U_k = k^{m-\frac{1}{2} -\theta}B^{m-\theta} T^{k-1}\, +\,  k^{m+\frac{1}{2} -\theta}B^{m+1-\theta} T^{k-1},
\end{equation}
and hence
$$
\norm{U_k(x)}^2\leq 2\bigl(k^{2m-2\theta-1}\norm{B^{m-\theta}T^{k-1} x}^2
+k^{2m-2\theta+1}\norm{B^{m+1 -\theta} T^{k-1} x}^2\bigr).
$$
According to (\ref{2SF}) this yields
$$
\sum_{k=1}^{\infty} \norm{U_k(x)}^2\,\leq \,2\bigl(\norm{x}^{2}_{T, m- \theta} 
\, +\, \norm{x}^{2}_{T, m+1-
\theta}\bigr)
$$
for any $x\in X$. (This is where assuming that $m\geq 1$ is useful.)
Applying Lemma \ref{2Indep} twice together with the estimate (\ref{3Sum}),
we deduce
$$
\sum_{k=1}^{\infty} k^\alpha \norm{CT^{k-1}(x)}^2\,
\leq K"^2 \norm{x}^2, \qquad x\in X,
$$
for an appropriate $K"\geq 0$, and hence the $\alpha$-admissibility of $C$.

When $m=0$, we have $\alpha\in(-1,1)$ and $\beta=-\alpha$. Then (\ref{3Weiss}) means that
for $\omega$ varying in $\Ddb$,
$$
\bigl(1-\vert \omega\vert^2\bigr)^{\frac{1-\alpha}{2}} C(I-\omega T)^{-1}
$$
is uniformly bounded.
By a series expansion, we have
$$
\bignorm{(I-\omega T)^{-1}}\,\leq \, \frac{c_0}{1-\vert\omega\vert}\,,\qquad \omega\in\Ddb,
$$
where $c_0$ is given by (\ref{2Pb}). Since $1-\vert\omega\vert^2\leq 2(1-\vert\omega\vert)$
when $\omega\in\Ddb$, we deduce that
$\bigl(1-\vert \omega\vert^2\bigr)\bignorm{(I-\omega T)^{-1}}\,\leq 2c_0$ on $\Ddb$.
This implies the existence of a constant $K\geq 0$ such that
$$
\bigl(1-\vert \omega\vert^2\bigr)^{\frac{3-\alpha}{2}}\bignorm{C(I-\omega T)^{-2}}\,\leq K,
\qquad \omega\in\Ddb.
$$
We may therefore apply the preceding part of the proof with $\beta=2-\alpha$ and $m=1$ to obtain that
$C$ is $\alpha$-admissible for $T$.
\end{proof}

Let $\alpha>-1$ be a real number and let $T\colon X\to X$
be a Ritt operator. Applying (\ref{2SF}) with $a =\frac{1+\alpha}{2}$
and Lemma \ref{2Indep}, we see that 
$C=(I-T)^{\frac{1+\alpha}{2}}$ is $\alpha$-admissible 
for $T$ if and only if $T$ satisfies a square function estimate.
The next result shows that $C=(I-T)^{\frac{1+\alpha}{2}}$ always
fulfils the Weiss condition (\ref{3Weiss}). Thus the square function estimate
assumption in Theorem \ref{3Main} cannot be omitted and is the `right'
assumption to make in that statement.

\begin{proposition}
Let $T$ be a Ritt operator, and let $\alpha,\beta>-1$ be real numbers such 
that $m=\frac{\alpha+\beta}{2}$ is a nonnegative integer. Then 
there exists a constant $K\geq 0$ such that
\begin{equation}\label{3Automatic}
\bigl(1-\vert \omega\vert^2\bigr)^{\frac{1+\beta}{2}}
\bignorm{(I-T)^{\frac{1+\alpha}{2}}(I-\omega T)^{-(m+1)}}\,\leq K,
\qquad \omega\in\Ddb.
\end{equation}
\end{proposition}

\begin{proof}
We let $B=I-T$ as before. Since $B$ is sectorial of type $<\frac{\pi}{2}$, there exists
$\gamma<\frac{\pi}{2}$ and a constant $K_\gamma$ such that $\xi-B$ is invertible
and $\vert\xi\vert\norm{(\xi-B)^{-1}}\leq K$
for any $\xi\in \Cdb^*$ with ${\rm Arg}(\xi)\geq\gamma$. 
Consider an auxiliary angle
$\nu\in\bigl(\frac{\pi}{2},\pi -\gamma\bigr)$. Clearly $z+B$ is invertible for any
$z\in \Sigma_\nu$. We claim that
\begin{equation}\label{3HLM45}
\sup\Bigl\{\vert z\vert^{\frac{1+\beta}{2}}\bignorm{B^{\frac{1+\alpha}{2}}(z+B)^{-(m+1)}}\,
:\, z\in\Sigma_\nu\Bigr\}\, <\infty\,.
\end{equation}
Indeed let $\Gamma_\gamma$ be the boundary of
$\Sigma_\gamma$ oriented counterclockwise. For any $z\in\Sigma_\nu$, we
have
$$
B^{\frac{1+\alpha}{2}}(z+B)^{-(m+1)}\,=\,\frac{1}{2\pi i}\,
\int_{\Gamma_\gamma} \frac{\xi^{\frac{1+\alpha}{2}}}{(z+\xi)^{m+1}}\,(\xi -B)^{-1}\, d\xi\,,
$$
see e.g. \cite{Ha} for details. We immediately deduce that
$$
\bignorm{B^{\frac{1+\alpha}{2}}(z+B)^{-(m+1)}}\,\leq\,\frac{K_\gamma}{2\pi}\,
\int_{\Gamma_\gamma}
\frac{\vert\xi\vert^{\frac{1+\alpha}{2}}}{\vert z+\xi\vert^{m+1}}\,\Bigl\vert \frac{d\xi}{\xi}\Bigr\vert\,.
$$
Writing $z=\vert z \vert e^{i\varphi}$ and changing $\xi$ into $\vert z\vert \xi$, we obtain that
the above integral is equal to
$$
\frac{\vert z \vert^{\frac{1+\alpha}{2}}}{\vert z \vert^{m+1}}\,
\int_{\Gamma_\gamma}
\frac{\vert\xi\vert^{\frac{1+\alpha}{2}}}{\vert e^{i\varphi} +\xi\vert^{m+1}}\,
\Bigl\vert \frac{d\xi}{\xi}\Bigr\vert\,.
$$
The latter integral remains bounded when $\varphi$ varies from $\pi - \nu$ to $\pi+\nu$. Moreover we have 
$\frac{1+\alpha}{2} -(m+1) = -\frac{1+\beta}{2}\,$, hence 
(\ref{3HLM45}) follows at once.

When $\omega\in\Ddb\setminus\{0\}$ and $\lambda=\frac{1}{\omega}$, we have
$$
\bigl(1-\vert \omega\vert^2\bigr)^{\frac{1+\beta}{2}}
B^{\frac{1+\alpha}{2}}(I-\omega T)^{-(m+1)} \,
=\,\frac{\lambda^{m+1}}{\vert\lambda\vert^{1+\beta}}\,
\bigl(\vert \lambda \vert^2 -1\bigr)^{\frac{1+\beta}{2}}
B^{\frac{1+\alpha}{2}}(\lambda -T)^{-(m+1)}.
$$
Further the norms of these operators are bounded when $\vert\omega\vert$ is away from $1$
(equivalently, when $\vert\lambda\vert\to\infty$).
Hence it suffices to show that
$$
\bigl\{\vert\lambda\vert^{m-\beta}
\bigl(\vert \lambda \vert^2 -1\bigr)^{\frac{1+\beta}{2}}
B^{\frac{1+\alpha}{2}}(\lambda -T)^{-(m+1)}\, :\, 1<\vert\lambda\vert<2\bigr\}
$$
is bounded. Writing $\vert \lambda \vert^2 -1=(\vert \lambda \vert  -1)(\vert \lambda \vert  +1)$, 
we see that this is equivalent to showing that the set
$$
\bigl\{ 
\bigl(\vert \lambda \vert -1\bigr)^{\frac{1+\beta}{2}}
B^{\frac{1+\alpha}{2}}(\lambda -T)^{-(m+1)}\, :\, 1<\vert\lambda\vert<2\bigr\}
$$
is bounded.
Since $T$ is a Ritt operator and $\nu>\frac{\pi}{2}$,  $(\lambda -T)^{-1}$ is bounded on
the set $\{\lambda\notin 1+\Sigma_\nu\}\cap\{1<\vert\lambda\vert<2\}$, by (\ref{3Ritt}). Hence 
$\bigl(\vert \lambda \vert -1\bigr)^{\frac{1+\beta}{2}}
B^{\frac{1+\alpha}{2}}(\lambda -T)^{-(m+1)}$ is bounded on that set.
It therefore suffices to show that 
\begin{equation}\label{3Bdd}
\bigl\{ 
\bigl(\vert \lambda \vert -1\bigr)^{\frac{1+\beta}{2}}
B^{\frac{1+\alpha}{2}}(\lambda -T)^{-(m+1)}\, :\, \lambda \in 
1+\Sigma_\nu,\ \vert\lambda\vert>1\bigr\}
\end{equation}
is bounded. Writing $\lambda=1+z$, we have $\lambda -T=z+B$
and $\vert \lambda \vert -1\leq \vert
z \vert$. Hence
$$
\bigl(\vert \lambda \vert -1\bigr)^{\frac{1+\beta}{2}}
\bignorm{B^{\frac{1+\alpha}{2}}(\lambda -T)^{-(m+1)}}
\,\leq \vert z \vert^{\frac{1+\beta}{2}}
\bignorm{B^{\frac{1+\alpha}{2}}(z +B)^{-(m+1)}}.
$$
The boundedness of (\ref{3Bdd}) therefore follows from (\ref{3HLM45}).
\end{proof}

To apply Theorem \ref{3Main}, one needs to know which Ritt operators 
satisfy a square function estimate. According to \cite[Thm. 8.1]{LM3}, Hilbert
space contractions have this property. This leads to the following.

\begin{corollary}\label{3Cont}
Let $X$ be a Hilbert space, let $T\colon X\to X$ be a contraction and
assume that $T$ is a Ritt operator. Then for any $\alpha>-1$ and
$\beta\in (-1,3)$ such that $m=\frac{\alpha+\beta}{2}$
is a nonnegative integer, and for any Banach space 
$Y$, a  bounded operator $C\colon X\to Y$ is
admissible for $T$ if and only if there exists a constant $K\geq 0$ such that
$$
\bigl(1-\vert \omega\vert^2\bigr)^{\frac{1+\beta}{2}}\bignorm{C(I-\omega T)^{-(m+1)}}\,\leq K,
\qquad \omega\in\Ddb.
$$
\end{corollary}

Note that by \cite[Prop. 8.2]{LM3}, there exist Ritt operators on Hilbert space with
a square function estimate (hence satisfying
the above corollary) without being similar to a contraction.

\medskip
\section{$R$-admissibility}
In this section we give an alternate set of results, similar to those
established in Sections 2 and 3, but using square functions different from the
ones in (\ref{2Admiss}) or (\ref{2SF}). The $\ell^2$-norms appearing in these 
formulas will be replaced by Rademacher averages. This approach is very
fruitful when dealing with non Hilbertian spaces, see in particular Corollaries
\ref{4App1} and \ref{4App2} below. The use of Rademacher averages in the context of admissibility
for a $c_0$-semigroup was initiated in \cite{LM0} in the framework of $L^p$-spaces
and then extended to a much broader context by Haak and Kunstmann \cite{H, HK1, HK2}.

We start with a little background on Rademacher sums.
Throughout we let $(\varepsilon_k)_{k\geq 1}$ 
be a sequence of independent Rademacher variables on some probability space
$(\M, d\Pdb)$ and for any Banach space $X$, we let ${\rm Rad}(X)$ denote the
closed subspace of the Bochner space $L^{2}(\M ;X)$ spanned 
by the set $\{\varepsilon_k\otimes x\, :\, k\geq 1, x\in X\}$. 
Thus for any
finite family $(x_k)_{k\geq 1}$ of elements of $X$,
\begin{equation}\label{4Rad}
\Bignorm{\sum_k \varepsilon_k\otimes x_k}_{{\rm Rad}(X)}\,
=\,\biggl(\int_{\footnotesize{\M}}\Bignorm{\sum_k 
\varepsilon_k(u) x_k}_{X}^{2}\, d\Pdb(u)\,\biggr)^{\frac{1}{2}}\,.
\end{equation}
Elements of ${\rm Rad}(X)$ are sums of convergent series of the form 
$\sum_{k=1}^{\infty}\varepsilon_k\otimes x_k\,$. Moreover when 
$X$ does not contain $c_0$ (as an isomorphic subspace), then a
series $\sum_{k} \varepsilon_k\otimes x_k\,$ converges in $L^{2}(\M ;X)$ if and only 
its partial sums are uniformly bounded \cite{Kwa}.

We let $B(X,Y)$ denote the Banach space of all bounded operators from $X$ into 
$Y$ and we recall that a subset $F\subset B(X,Y)$ is called $R$-bounded provided that there is 
a constant $K\geq 0$ such that 
$$
\Bignorm{\sum_k\varepsilon_k\otimes V_k(x_k)}_{{\rm Rad}(Y)}\,
\leq\,K
\Bignorm{\sum_k\varepsilon_k\otimes x_k}_{{\rm Rad}(X)}.
$$
for any finite families
$(V_k)_k$ in $F$ and $(x_k)_k$ in $X$. See \cite{CPSW} where this notion was thoroughly 
studied for the first time.

We now turn to admissibility. As before we consider a power bounded operator
$T\colon X\to X$ and we let $\alpha>-1$. 
We say that an operator $C\colon X\to Y$ is 
$\alpha$-$R$-admissible for $T$ if the series
$\sum_{k\geq 1} k^{\frac{\alpha}{2}}\varepsilon_k\otimes C T^{k-1}(x)\,$ converges in
${\rm Rad}(Y)$ for any $x\in X$ and there is a constant $M\geq 0$
such that
\begin{equation}\label{4Rad-adm}
\Bignorm{\sum_{k=1}^{\infty}k^{\frac{\alpha}{2}}\varepsilon_k\otimes CT^{k-1}(x)}_{{\rm Rad}(Y)}\,
\leq M\norm{x},
\qquad x\in X.
\end{equation}
When $Y$ is a Hilbert space, $\bignorm{\sum_k \varepsilon_k\otimes y_k}_{{\rm Rad}(Y)}$ 
is equal to $\bigl(\sum_k\norm{y_k}^2\bigr)^{\frac{1}{2}}$ for any $(y_k)_k$ in $Y$. Thus
in this
case, $\alpha$-$R$-admissibility (\ref{4Rad-adm}) 
coincides with $\alpha$-admissibility (\ref{2Admiss}). However in general, these 
two notions are quite different.

We will make use of a few notions from Banach space theory such as cotype and $K$-convexity, 
for which we refer e.g. to \cite{DJT}. We recall that $X$ being $K$-convex means that 
the space ${\rm Rad}(X)^*$ is canonically
isomorphic to ${\rm Rad}(X^*)$.

The following is the `$R$-analog' of Proposition \ref{2Easy}. 

\begin{proposition}\label{4Easy}
Let $X,Y$ be Banach spaces and assume that $Y$ is $K$-convex.
Let $\alpha,\beta>-1$ be real numbers such that $m=\frac{\alpha+\beta}{2}$ is a nonnegative integer.
If $C$ is $\alpha$-$R$-admissible for $T$, then the set
$$
\Bigl\{\bigl(1-\vert \omega\vert^2\bigr)^{\frac{1+\beta}{2}}
C(I-\omega T)^{-(m+1)}\, :\, \omega\in\Ddb\Bigr\}\,\subset B(X,Y)
$$
is $R$-bounded.
\end{proposition}

\begin{proof}
Let $(\omega_k)_k$ be a finite family of $\Ddb$, 
let $V_k= \bigl(1-\vert \omega_k\vert^2\bigr)^{\frac{1+\beta}{2}}
C(I-\omega_k T)^{-(m+1)}$ for any $k$ and let $(x_k)_k$ be a finite family 
of $X$. By the $K$-convexity assumption, there exists a finite family
$(z_k)_k$ of $Y^*$ such that 
\begin{equation}\label{4HB}
\Bignorm{\sum_k\varepsilon_k\otimes V_k(x_k)}_{{\rm Rad}(Y)}\, = 
\sum_k \bigl\langle z_k, V_k(x_k)\bigr\rangle
\end{equation}
and
\begin{equation}\label{4K-Conv}
\Bignorm{\sum_k\varepsilon_k\otimes z_k}_{{\rm Rad}(Y^*)}\,\leq K_0,
\end{equation}
where $K_0$ is a numerical constant only depending on $Y$.

For any $k$ and $j\geq 1$, set
$$
a_{jk}= \frac{1}{m!}\,\bigl(1-\vert \omega_k\vert^2\bigr)^{\frac{1+\beta}{2}} c_j \omega_k^{j-1},
$$
where $c_j$ is defined by (\ref{2ck}). The computation at the end of the proof 
of Proposition \ref{2Easy} shows that the $\sum_j\vert a_{jk}\vert^2$
are uniformy bounded, so that we have a constant $K_1\geq 0$ such that
\begin{equation}\label{4a}
\Bigl(\sum_j\vert a_{jk}\vert^2\Bigr)^{\frac{1}{2}}\,\leq K_1,\qquad k\geq 1.
\end{equation}
According to (\ref{2Exp}),
$$
V_k(x_k) = \sum_{j=1}^{\infty} a_{jk} j^{\frac{\alpha}{2}} CT^{j-1}(x_k)
$$
for any $k$, hence 
\begin{align*}
\sum_k \bigl\langle z_k, V_k(x_k)\bigr\rangle\, & =\,\sum_{j,k} \bigl\langle a_{jk} z_k, j^{\frac{\alpha}{2}} CT^{j-1}(x_k)\bigr\rangle\\
&\leq \Bignorm{\sum_{j,k} a_{jk}\,\varepsilon_k\otimes\varepsilon_j\otimes z_k}_{{\rm Rad}({\rm Rad}(Y^*))}\,
\Bignorm{\sum_{j,k} j^{\frac{\alpha}{2}}\,\varepsilon_k\otimes\varepsilon_j
\otimes CT^{j-1}(x_k)}_{{\rm Rad}({\rm Rad}(Y))}\,.
\end{align*}
Using (\ref{4Rad}) and the $\alpha$-$R$-admissibility assumption, we obtain that
\begin{align*}
\Bignorm{\sum_{j,k} j^{\frac{\alpha}{2}}\,\varepsilon_k\otimes\varepsilon_j
\otimes CT^{j-1}(x_k)}_{{\rm Rad}({\rm Rad}(Y))}^2\,
& =\int\int \Bignorm{\sum_{j,k}
\varepsilon_k(u)\varepsilon_j(v) j^{\frac{\alpha}{2}} CT^{j-1}(x_k)}_{Y}^{2}\, d\Pdb(v)d\Pdb(u)\\
&= \int_{\footnotesize{\M}}\Bignorm{\sum_{j}j^{\frac{\alpha}{2}} \,\varepsilon_j\otimes CT^{j-1}\Bigl(
\sum_k\varepsilon_k(u)x_k\Bigr)}_{{\rm Rad}(Y)}^2\, d\Pdb(u)\\
& \leq M^2\,\Bignorm{\sum_k \varepsilon_k\otimes x_k}_{{\rm Rad}(X)}^2.
\end{align*}
Further since $Y$ is $K$-convex, the space $Y^*$ has a finite cotype. Hence it follows from \cite[Cor. 3.4]{KW}
that for some constant $K_2\geq 0$ (only depending on $Y^*$), we have
$$
\Bignorm{\sum_{j,k} a_{jk}\,\varepsilon_k\otimes\varepsilon_j\otimes z_k}_{{\rm Rad}({\rm Rad}(Y^*))}\,
\leq\,K_2\sup_k\Bigl(\sum_j\vert a_{jk}\vert^2\Bigr)^{\frac{1}{2}}
\,\Bignorm{\sum_k\varepsilon_k\otimes z_k}_{{\rm Rad}(Y^*)}.
$$
Inserting (\ref{4HB}), (\ref{4K-Conv}) and (\ref{4a}) in the above computation, we obtain 
$$
\Bignorm{\sum_k\varepsilon_k\otimes V_k(x_k)}_{{\rm Rad}(Y)}\,\leq\, M K_0 K_1K_2 
\Bignorm{\sum_k\varepsilon_k \otimes x_k}_{{\rm Rad}(X)}\,,
$$
which yields the result.
\end{proof}

We say that $T\colon X\to X$ is an $R$-Ritt operator if the two sets
$$
\bigl\{T^k\, :\, k\geq 0\bigr\}\qquad\hbox{and}\qquad \bigl\{k(T^k - T^{k-1})\, :\, k\geq 1\bigr\}
$$
are both $R$-bounded. This is a strengthening of (\ref{2Pb}) and (\ref{2Ritt}).
This notion was introduced by Blunck \cite{Bl1,Bl2}, see also \cite{LM3} for 
information.

For any Ritt operator $T$ and any $a>0$, let us consider the abstract square function
${\rm SF}_{T,a}$ defined by 
$$
{\rm SF}_{T,a}(x)\,=\,\Bignorm{\sum_{k=1}^{\infty} k^{a-\frac{1}{2}}\,\varepsilon_k\otimes T^{k-1}(I-T)^{a}(x)
}_{{\rm Rad}(X)},\qquad x\in X
$$
(with the convention that ${\rm SF}_{T,a}(x)=\infty$ if the above series diverges).
These square functions coincide with the ones defined in \cite[Section 6]{ALM}.
The following is the `$R$-analog' of Lemma \ref{2Indep}.

\begin{lemma}\label{4Indep} \cite[Thm. 6.1]{ALM} Assume that 
$X$ is reflexive with a finite cotype and let $T\colon X\to X$ be an $R$-Ritt operator.
Then the square functions
${\rm SF}_{T,a}$ are pairwise equivalent.
\end{lemma}

In the sequel we say that $T$ satisfies an $R$-square function estimate if 
there exists a constant $\kappa\geq 0$ such that
\begin{equation}\label{4SFE}
\Bignorm{\sum_{k=1}^{\infty} k^{\frac{1}{2}}
\,\varepsilon_k\otimes \bigl(T^k(x)-T^{k-1}(x)\bigr)}_{{\rm Rad}(X)}
\,\leq \kappa\norm{x},
\qquad x\in X.
\end{equation}

We can now state the main result of this section, which is the analog of Theorem
\ref{3Main} for $R$-admissibility.

\begin{theorem}\label{4Main}
Let $X,Y$ be Banach spaces and assume that $X$ is 
reflexive and has a finite cotype. Let 
$T\colon X\to X$ be an $R$-Ritt operator satisfying
the $R$-square function estimate (\ref{4SFE}). Let $\alpha>-1$ and
$\beta\in (-1,3)$ be real numbers such that $m=\frac{\alpha+\beta}{2}$
is a nonnegative integer. Then a bounded operator $C\colon X\to Y$ is
$\alpha$-$R$-admissible for $T$ if the set 
\begin{equation}\label{4Weiss}
\Bigl\{\bigl(1-\vert \omega\vert^2\bigr)^{\frac{1+\beta}{2}}
C(I-\omega T)^{-(m+1)}\, :\, \omega\in\Ddb\Bigr\}\,\subset B(X,Y)
\end{equation}
is $R$-bounded.
\end{theorem}

\begin{proof} This is a simple adaptation of the proof of Theorem \ref{3Main} so we will be
deliberately sketchy. We consider $B=I-T$, we may assume that it has a dense range, 
and we take some $\theta\in (0,1)$ such that 
$\theta<\frac{1+\beta}{2}<1+\theta$. Under the assumption that (\ref{4Weiss}) is
$R$-bounded, the computation in the proof of Theorem \ref{3Main} shows that the set
$\bigl\{t^{\frac{1+\beta}{2}} C(t+B)^{-(m+1)}\, :\, t>0\bigr\}$ is $R$-bounded.
Then the computation from \cite{HL} leading to Lemma \ref{3HLM} shows that in turn, the set
\begin{equation}\label{4R}
\bigl\{t^{\frac{1+\alpha}{2} -m} C B^{-m}\varphi_\theta(tB)\, :\, t>0\bigr\}
\quad\hbox{is}\  R\hbox{-bounded}.
\end{equation}

Assume that $m\geq 1$ and let
$$
V_k=k^{\frac{1+\alpha}{2} -m} CB^{-m}\varphi_{\theta}(kB)
$$
for any $k\geq 1$. According to the proof of Theorem \ref{3Main}, we have
$$
k^{\frac{\alpha}{2}}CT^{k-1} \, =\, V_k(k^{m-\frac{1}{2} -\theta}B^{m-\theta} T^{k-1})\, +\,  
V_k(k^{m+\frac{1}{2} -\theta}B^{m+1-\theta} T^{k-1})
$$
for any $k\geq 1$. Then for any $x\in X$, and for any integers $n_1>n_0\geq 0$, we have
\begin{align*}
\Bignorm{\sum_{k=n_0}^{n_1}k^{\frac{\alpha}{2}}\,\varepsilon_k\otimes CT^{k-1}(x)}_{{\rm Rad}(Y)}\,
\leq\, & \Bignorm{\sum_{k=n_0}^{n_1} \varepsilon_k\otimes V_k
\bigl(k^{m-\frac{1}{2} -\theta}B^{m-\theta} T^{k-1}(x)\bigr)}_{{\rm Rad}(Y)}\\ & 
+\,\Bignorm{\sum_{k=n_0}^{n_1} \varepsilon_k\otimes V_k
\bigl(k^{m+\frac{1}{2} -\theta}B^{m+1-\theta} T^k(x)\bigr)}_{{\rm Rad}(Y)}\,.
\end{align*}
By the definition of $R$-boundedness and (\ref{4R}), this implies an estimate 
\begin{align*}
\Bignorm{\sum_{k=n_0}^{n_1}k^{\frac{\alpha}{2}}\,\varepsilon_k \otimes CT^{k-1}(x)}_{{\rm Rad}(Y)}\,
\leq\,  K\biggl( &
\Bignorm{\sum_{k=n_0}^{n_1} \varepsilon_k\otimes 
k^{m-\frac{1}{2} -\theta}B^{m-\theta} T^{k-1}(x)}_{{\rm Rad}(X)}\\ & 
+\,\Bignorm{\sum_{k=n_0}^{n_1} \varepsilon_k\otimes 
k^{m+\frac{1}{2} -\theta}B^{m+1-\theta} T^{k-1}(x)}_{{\rm Rad}(X)}\biggr).
\end{align*}
Taking into account our square function estimate assumption and Lemma \ref{4Indep}, this shows that
$C$ is $\alpha$-$R$-admissible for $T$, with
$$
\Bignorm{\sum_{k=0}^{\infty}k^{\frac{\alpha}{2}}\,\varepsilon_k\otimes CT^{k-1}(x)}_{{\rm Rad}(Y)}\,
\leq K\bigl({\rm SF}_{T, m-\theta}(x) \, +\, {\rm SF}_{T, m+1-\theta}(x)\bigr).
$$

The case $m=0$ can be deduced from the case $m=1$ as in the proof of Theorem \ref{3Main}
\end{proof}

The main motivation for considering $R$-admissibility lies in the existence of classical classes of
Ritt operators satisfying an $R$-square function estimate (\ref{4SFE}). We refer the reader to
\cite{LM3} for various results concerning this property, and its relationships with the
so-called $H^{\infty}(B_\gamma)$ functional calculus that we are going to use in the next statement. 
Combining Proposition \ref{4Easy} and Theorem \ref{4Main}
with \cite[Corollary 6.9]{LM3}, we obtain the following.

\begin{corollary}\label{4App1} 
Let $X,Y$ be Banach spaces, assume that $X$ is reflexive and has a finite cotype,
and assume that $Y$ is $K$-convex.
Let $T\colon X\to X$ be a Ritt operator which admits a bounded 
$H^{\infty}(B_\gamma)$ functional calculus for some $\gamma<\frac{\pi}{2}$.
Then for any $\alpha>-1$ and
$\beta\in (-1,3)$ such that $m=\frac{\alpha+\beta}{2}$
is a nonnegative integer, a bounded operator $C\colon X\to Y$ is
$\alpha$-$R$-admissible for $T$ if and only if the set (\ref{4Weiss}) is $R$-bounded.
\end{corollary}

On $L^q$-spaces, Rademacher averages are equivalent to genuine square functions. 
Indeed whenever $(y_k)_k$ is a finite sequence of $L^q$ (with $q<\infty$), then
$\bignorm{\sum_k \varepsilon_k\otimes y_k}_{{\rm Rad}(L^q)}$ 
is equivalent to $\bignorm{\bigl(\sum_k \vert y_k \vert^2\bigr)^{\frac{1}{2}}}_{L^q}$.
Thus when $Y$ is an $L^q$-space, the above statements
can be written is a more concrete form. Here is an illustration.

\begin{corollary}\label{4App2} 
Let $(\Omega,\mu)$ be a measure space and let
$T\colon L^p(\Omega)\to L^p(\Omega)$ be a positive contraction, with $1<p<\infty$. 
Assume that $T$ is a Ritt operator. Let $(\Omega',\mu')$ be another 
measure space and let $1<q<\infty$. Then for any $\alpha>-1$ and
$\beta\in (-1,3)$ such that $m=\frac{\alpha+\beta}{2}$
is a nonnegative integer, and for any operator $C\colon L^p(\Omega)\to L^q(\Omega')$, 
the set (\ref{4Weiss}) is $R$-bounded if and only if there
exist a constant $M\geq 0$ such that
\begin{equation}\label{4Lq}
\Bignorm{\Bigl(\sum_{k=0}^{\infty}(k+1)^\alpha\bigl\vert CT^k(x)\bigr\vert^2\Bigr)^{\frac{1}{2}}}_q\,
\leq\, M\norm{x}_p,
\qquad x\in L^p(\Omega).
\end{equation}
\end{corollary}

\begin{proof}
By \cite[Prop. 3.2]{LMX}, $T$ admits a bounded 
$H^{\infty}(B_\gamma)$ functional calculus for some $\gamma<\frac{\pi}{2}$.
Further $X=L^p(\Omega)$ is reflexive with a finite cotype and $Y=L^q(\Omega')$
is $K$-convex. Hence Corollary \ref{4App1} ensures
that the set (\ref{4Weiss}) is $R$-bounded if and only 
if $C$ is $\alpha$-$R$-admissible for $T$. According to the discussion
before the statement of Corollary \ref{4App2}, this is equivalent to 
(\ref{4Lq}).
\end{proof}

\end{document}